\documentclass[12pt,a4paper]{article}
\usepackage{amssymb,amsmath,amsthm}
\usepackage{graphicx}
\usepackage{hyperref}
\def \ex {\mathrm{ex}}
\def \rex {\mathrm{rex}}
\def \regex {\mathrm{regex}}

\theoremstyle{plain}
\newtheorem{theorem}{Theorem}[section]
\newtheorem{lemma}[theorem]{Lemma}

\newtheorem{conjecture}[theorem]{Conjecture}
\newtheorem{proposition}[theorem]{Proposition}

\theoremstyle{definition}

\newtheorem{claim}[theorem]{Claim}

\newcommand\cref[1]{Corollary~\ref{cor:#1}}

\title{Some exact results for regular Tur\'an problems}
\author{D\'aniel Gerbner$^1$ \and Bal\'azs Patk\'os$^{1,2}$ \and Zsolt Tuza$^{1,3}$ \and M\'at\'e Vizer$^1$ \\ \small $^1$ Alfr\'ed R\'enyi Institute of Mathematics\\ \small $^2$ Lab. of Combinatorial and Geometric Structures, Moscow Inst. of Physics and Technology\\ \small $^3$ Department of Computer Science and Systems
 Technology, University of Pannonia%, Hungarian Academy of Sciences
}
\date{}

\begin{document}

\maketitle

\begin{abstract}
As a variant of the famous Tur\'an problem, we study $\rex(n,F)$, the maximum number of edges that an $n$-vertex \textit{regular} graph can have without containing a copy of $F$. We determine $\rex(n,K_{r+1})$ for all pairs of integers $r$ and large enough $n$. For every tree $T$, we determine $\rex(n,T)$ for every $n$ large enough.
%that were unsolved before. 
%We determine the asymptotics of $rex(n,F)$ for a wide class of 3-chromatic graphs $F$ and we also settle the problem for trees.
\end{abstract}

\section{Introduction}

The Tur\'an number, $\ex(n,F)$ is the maximum number of edges that an $n$-vertex graph can have without containing a copy of $F$. Since the introduction of the problem by Tur\'an \cite{T}, the area attracted the attention of many researchers. The order of magnitude of $\ex(n,F)$ is known unless $F$ is bipartite. There have been lots of generalizations, strengthenings of results and notions related to the Tur\'an number. A recent one is the singular Tur\'an number due to Caro and Tuza \cite{ct}. In a follow-up paper \cite{gptv}, the present authors introduced the regular Tur\'an number $\rex(n,F)$, the maximum number of edges that a \textit{regular} graph on $n$ vertices can have without containing a copy of $F$. Sometimes it is more convenient to work with $\regex(n,F)$, the maximum number $r$ such that there exists an $r$-regular, $F$-free graph on $n$ vertices. Clearly, we have $2\,\rex(n,F)=n \cdot \regex(n,F)$. 

Recent papers by Caro and Tuza \cite{ct2} and independently by Cambie, de Verclos and Kang \cite{cdk} started investigating these parameters. Both papers determined $\rex(n,F)$ asymptotically if $F$ has chromatic number at least 4, and obtained asymptotic results for several graphs with chromatic number 3. However, not many exact results are known. By exact result, we mean that $\rex(n,F)$ is determined for \emph{every} $n$ large enough with respect to $F$. Note that in several cases, the exact value of $\rex(n,F)$ is an immediate corollary of knowing $\ex(n,F)$ for $n$ satisfying some divisibility condition, but not for other values of $n$.

Caro and Tuza \cite{ct2} determined $\rex(n,F)$ for $n$ large enough in case $F$ is $K_3$, $C_5$, or $K_4$ minus an edge. A theorem of Cambie, de Verclos and Kang \cite{cdk} implies an exact result for $n$ large enough for every 3-chromatic graph which contains a triangle and is \textit{edge-critical}, i.e.\ has an edge, such that deleting that edge we obtain a bipartite graph (using a result of Simonovits \cite{sim}, which determines the Tur\'an number of edge-critical graphs exactly, for $n$ large enough). In this paper we will obtain exact results for cliques and trees.

Let us write $\delta(n,r)$ to denote the minimum degree of the $r$-partite Tur\'an graph on $n$ vertices, i.e.\ if $n=qr+s$ with $0\le s<r$, then $\delta(n,r)=n-q$ if $s=0$ and $\delta(n,r)=n-q-1$ otherwise.

\begin{theorem}[Caro and Tuza \cite{ct2}]
   (i) If\/ $n$ is a multiple of\/ $r$, then\/ $\rex(n, K_{r+1}) = \ex(n, K_{r+1})$.

(ii) If\/ $n$ is not a multiple of\/ $r$, then\/ $\rex(n, K_{r+1}) = \ex(n, K_{r+1}) - \Theta(n)$ as\/
$n \rightarrow \infty$.

(iii) More exactly, if\/ $n = qr + s$ with\/ $1 \le s \le r -2$, and at least one of\/
$r - s$ or\/ $q$ is even, 

then\/
$\regex(n, K_{r+1}) = \delta(n,r)$.

(iv) For\/ $n = 3q + s$ with\/ $s = 0, 1, 2$ we have\/ $\regex(n, K_4) = 2q$.
\end{theorem}
    
    Our first result determines $\regex(n,K_{r+1})$ in all remaining cases.
    
    \begin{theorem}\label{complete}
        Let\/ $n=qr+s$ with\/ $r\ge 4$,  and\/ $q$ be large enough.
        
        (i) If\/ $q$ and\/ $r-s$ are odd, then\/ $\regex(n,K_{r+1})=\delta(n,r)-1$.
        
        (ii) If\/ $q$ and\/ $r$ are even with\/ $s=r-1$, then\/ $\regex(n,K_{r+1})=\delta(n,r)-2$.
        
        (iii) If\/ $q$ is even and\/ $r$ is odd with\/ $s=r-1$, then\/ $\regex(n,K_{r+1})=\delta(n,r)-1$.
    \end{theorem}

           %\begin{theorem}
           %Let $F$ be a graph with chromatic number 3 and let $k$ be the largest integer such that there exists a blow-up of $C_k$ that contains $F$. Then $rex(n,F)=(\frac{1}{k+2}+o(1))n^2$.
           %\end{theorem}

Let us turn our attention to trees. As Caro and Tuza \cite{ct} showed, for graphs with chromatic number at least $r+1\ge 4$, the difference between (the asymptotics of) ordinary and regular Tur\'an numbers is that some small changes are needed to the Tur\'an graph to obtain a slightly smaller, regular complete $r$-partite graph. For graphs with chromatic number 3, the situation is way more complicated in the regular case. We will show that for trees this is the simpler problem. The Erd\H os--S\'os conjecture \cite{ErSo} states that for any tree $T$ with $t$ vertices, we have $\ex(n,T)\le (t-2)n/2$, and it is still only known for some classes of graphs. However, $\rex(n,T)\le (t-2)n/2$ trivially follows from the well-known fact that a graph with minimum degree at least $t-1$ contains every tree on $t$ vertices. Note that if $(t-1)$ divides $n$, then this bound is sharp, as shown by $n/(t-1)$ vertex-disjoint copies of $K_{t-1}$. With the next theorem we determine $\rex(n,T)$ for every tree if $n$ is large enough. We say that a tree is \textit{almost-star} if in its proper 2-coloring, one of the classes consists of
at most two vertices.

\begin{theorem}\label{treethm} Let\/ $T$ be a tree on\/ $t$ vertices and\/ $n > n_0(T)$. Then
\begin{displaymath}
\regex(n,T)=
\left\{ \begin{array}{l l}
t-2 & \textrm{if\/ $t-1$ divides\/ $n$ or\/ $T$ is a star and\/ $t$ or\/ $n$ is even},\\
t-3 & \textrm{if\/ the above does not hold, and either\/ $t$ is odd, or\/ $t-2$ divides\/ $n$,}\\ & \textrm{ or\/ $T$ is a star or\/ $T$ is an almost-star and\/ $n$ is even}.\\
t-4 & \textrm{otherwise}.\\
\end{array}
\right.
\end{displaymath}
\end{theorem}

We remark that for very small trees $T$, the above theorem gives $\regex(n,T)=0$. Indeed, while $\ex(n,F)=0$ if and only if $n=1$ or $F=K_2$, there are more cases when $\rex(n,F)=0$. One kind of example is when $n$ is small, for example $\rex(5,C_5)=0$, as the only regular graphs on 5 vertices are $K_5$ and $C_5$. The other, more interesting examples are graphs $F$ where $\rex(n,F)=0$ holds for infinitely many $n$. We characterize these graphs below.

\begin{proposition}\label{nulla}
For a given graph\/ $F$, we have $\rex(n,F)=0$ for infinitely many\/ $n$ if and only if every component of\/ $F$ is\/ $P_2$ or\/ $P_3$, except for at most one component that is\/ $P_4$.
\end{proposition}

\section{Proofs}
We will use the following result of Brouwer \cite{br}. 

\begin{theorem}\label{partite+}
If\/ $H$ is a\/ $K_{r+1}$-free graph on\/ $n$ vertices which is not\/ $r$-partite, then\/ $H$ has at most\/ $t(n,r)-\lfloor n/r\rfloor+1$ edges, assuming\/ $n\ge 2r+1$.\end{theorem}

Hanson and Toft \cite{ht} also 
characterized the extremal graphs (the same result was independently obtained in \cite{afgs,kp}).

\begin{proof}[Proof of Theorem \ref{complete}]
Let us start with the upper bounds. Clearly, Tur\'an's theorem implies $\regex(n,K_{r+1})\le \delta(n,r)$. Observe that if the conditions of (i) hold, then $n$, $q$ and $\delta(n,r)$ are all odd, therefore there cannot exist a $\delta(n,r)$-regular graph on $n$ vertices and thus we have $\regex(n,K_{r+1})\le \delta(n,r)-1$.

In the cases (ii) and (iii), it is enough to show that there does not exist $\delta(n,r)$-regular $K_{r+1}$-free graphs on $n$ vertices. Indeed, that would imply $\regex(n,K_{r+1})\le \delta(n,r)-1$ and if the conditions of (ii) hold, then $n$ and $\delta(n,r)-1$ are both odd, therefore there is no $K_{r+1}$-free $(\delta(n,r)-1)$-regular graph on $n$ vertices.

So let $G$ be a regular $K_{r+1}$-free graph on $n=(q+1)r-1$ vertices. If $G$ is not $r$-partite, then Theorem \ref{partite+} implies $e(G)\le t(n,r)-\lfloor n/r\rfloor+1$. As $t(n,r)=n\delta(n,r)/2+q/2$, we obtain that  $e(G)$ is strictly smaller than $n\delta(n,r)/2$ if $q\ge 2$ and thus the regularity of $G$ must be smaller than $\delta(n,r)$. Finally, assume $G$ is $r$-partite and thus contained in a complete $r$-partite graph $K$. If $K$ is not a Tur\'an graph, then the minimum degree of $K$ is strictly smaller than $\delta(n,r)$ and we are done. If $G$ is a $\delta(n,r)$-regular subgraph of a Tur\'an graph $T(n,r)$, then $G$ must be obtained from $T(n,r)$ by removing edges only incident to vertices of maximum degree. But because of $s=r-1$, these vertices form an independent set, so such $G$ cannot exist.

\medskip

For the lower bounds we need constructions. As in \cite{ct2}, we will make use of Dirac's theorem that states that if the minimum degree of $G$ is at least $|V(G)|/2$, then $G$ contains a Hamiltonian cycle. In all cases, we will obtain our regular graph $G$ by removing edges from a Tur\'an graph $T(n,r)$. Note that $r-s$ is the number of parts in $T(n,r)$ that contain vertices of maximum degree, so the total number of such vertices is $q(r-s)$. Let $H_1$ and $H_2$ denote the subgraphs of $T(n,r)$ induced on the minimum and maximum degree vertices, respectively.

\medskip 

First we will provide good constructions for $(i)$. So let $q$ and $r-s$ be odd.

\smallskip 

$\bullet$ Assume first that $r-s,\ s>1$. Then both $H_1$ and $H_2$ satisfy the conditions of Dirac's theorem, so we can remove a Hamiltonian cycle from $H_2$ and every other edge of a Hamiltonian cycle of $H_1$ (note that the number of vertices in $H_1$ is $n-q(r-s)$, thus even by our assumptions) to make all degrees $\delta(n,r)-1$. 

\smallskip 

$\bullet$ If $s=1$, then let us remove an edge incident to the $q+1$ vertices of minimal degree each such that the other end vertices of these edges are all distinct and are as evenly distributed among the other parts of $T(n,r)$ as possible. As $q+1$ is even, $r-1\ge 3$, there exists a $1$-factor on these vertices, which we also remove. Finally, the remaining maximum-degree vertices span a subgraph of $H_2$ that satisfies the conditions of Dirac's theorem, and therefore we can remove a Hamiltonian cycle from that subgraph, to make all degrees $\delta(n,r)-1$. 

\smallskip 

$\bullet$ If $r-s=1$, then let us remove two edges incident to each maximum-degree vertex such that all $2q$ other endpoints are distinct and as evenly distributed among the other parts of $T(n,r)$ as possible. If $q$ is at least 2 (so large enough), then the subgraph of $H_1$ spanned by the remaining min-degree vertices satisfies the conditions of Dirac's theorem, so we can remove every other edge of a Hamiltonian cycle (as $n-q-2q$ is even in this case) to obtain a $(\delta(n,r)-1)$-regular graph. This completes the proof of (i).

\medskip 

To obtain a construction for (ii), it is enough to define a subset $E$ of the edges of $T(n,r)$ with every maximum-degree vertex of $T(n,r)$ being incident to 3 edges in $E$ and every minimum-degree vertex of $T(n,r)$ being incident to 2 edges of $E$. Indeed, then removing $E$ from $T(n,r)$ results in a $(\delta(n,r)-2)$-regular subgraph of $T(n,r)$. As $s=r-1$ implies that the maximum-degree vertices of $T(n,r)$ form an independent set, we must add $3q$ edges $e_1,e_2,\dots,e_{3q}$ to $E$ with $e_{3j-2},e_{3j-1},e_{3j}$ being incident to the $j$th maximum-degree vertex. The other endpoints of these edges are all distinct (this is possible as $(r-1)(q+1)\ge 3q$ holds by the assumption $r \ge 4$) and are as evenly distributed among the minimum-degree parts of $T(n,r)$ as possible. Let $V_1$ denote the set of minimum-degree vertices that are not incident to any of $e_1,e_2,\dots,e_{3q}$ and let $V_2$ denote the remaining minimum-degree vertices. It is easy to verify that both $T(n,r)[V_1]$ and $T(n,r)[V_2]$ satisfy the conditions of Dirac's theorem, so there exist Hamiltonian cycles $C_1$ and $C_2$. Adding all edges of $C_1$ to $E$ and every other edge of $C_2$ to $E$ ($3q$ is even) finishes the definition of $E$.

\medskip 

Finally, to prove (iii) let us suppose $q$ is even, $r$ is odd and $s=r-1$ hold. Similarly as in (ii), it is enough to define a subset $E$ of edges of $T(n,r)$ with every maximum-degree vertex of $T(n,r)$ being incident to $2$ edges in $E$ and every minimum-degree vertex of $T(n,r)$ being incident to 1 edge of $E$. This time  we must add $2q$ edges $e_1,e_2,\dots,e_{2q}$ to $E$ with $e_{2j-1},e_{2j}$ being incident to the $j$th maximum-degree vertex. The other endpoints of these edges are all distinct (this is possible as $(r-1)(q+1)\ge 2q$ holds by the assumption $r \ge 4$) and are as evenly distributed among the minimum-degree parts of $T(n,r)$ as possible. Let $V$ denote the vertices not incident to any of $e_1,e_2,\dots,e_{2q}$. Again, $T(n,r)[V]$ satisfies the conditions of Dirac's theorem, therefore we can add every other edge of the Hamiltonian cycle (note that $|V|=n-3q=r(q+1)-1-3q$ is even) to $E$. This completes the proof of (iii) and the theorem as well.
\end{proof}

Let us continue with the proof of Proposition \ref{nulla}.

\begin{proof}[Proof of Proposition \ref{nulla}] If $F$ has a vertex of degree at least three, then $C_n$ is a regular $F$-free graph for every $n>2$. If $F$ contains a cycle $C_k$, then again $C_n$ is a regular $F$-free graph for $k\neq n>2$. If $F$ contains $P_5$, then we write $n=4a+3b$ (every $n>5$ can be written this way) and take $a$ copies of $C_4$ and $b$ copies of $C_3$ to obtain a regular $F$-free graph on $n$ vertices. Therefore, if $\rex(n,F)=0$ for infinitely many $n$, then every component of $F$ is $P_2$, $P_3$ or $P_4$. If there are at least two components that are $P_4$'s, then we write $n>2$ as $3a+p$ with $0\le p<3$, and take $a-1$ copies of $C_3$ and a $C_{3+p}$ to obtain a regular $F$-free graph on $n$ vertices.

Now assume every component of $F$ is $P_2$ or $P_3$, except for at most one which is $P_4$. Let us assume $n$ is large enough, not divisible by 2 or 3, and let $G$ be a $d$-regular $F$-free graph on $n$ vertices. Observe that $d>1$ because of the parity of $n$.

Let us assume first that $G$ does not contain a $P_4$. Then $d\le 2$, thus $d=2$, and $G$ is a union of cycles. Moreover, the length of these cycles is at most three, thus $G$ is the union of triangles, which contradicts our assumption on $n$.

Now assume that there is a copy of $P_4$ in $G$ and let $k$ be the largest integer such that there are $k$ vertex-disjoint copies of $P_3$ in the other $n-4$ vertices of $G$. Let $A$ be the set of the $3k+4$ vertices in those $k$ copies of $P_3$ and one copy of $P_4$. Then for every vertex $v$ in $V(G)\setminus A$, at least $d-1$ edges incident to $v$ go to $A$. In particular, at least half the edges are incident to vertices in $A$, which by the regularity of $G$ would mean that there are at most $6k+8$ vertices. Now let $\ell$ be the number of components of $F$. If $n>6\ell+8$, then $G$ contains a $P_4$ and more than $\ell$ copies of $P_3$ vertex-disjointly, thus contains $F$, a contradiction which finishes the proof.
\end{proof}

Before proving Theorem \ref{treethm}, we need some lemmas.

\begin{lemma}\label{nemklikk}
For any non-star tree\/ $T$ on\/ $t$ vertices, the only\/ $T$-free\/ $(t-2)$-regular graphs are the vertex-disjoint unions of copies of\/ $K_{t-1}$.
\end{lemma}

\begin{proof} Let $G$ be a connected component of a $T$-free $(t-2)$-regular graph. We show that $G$ is a clique. If $G$ does not contain three vertices $a,b,c$ such that $ab$ and $bc$ are edges of $G$, but $ac$ is not, then either $G$ is a matching (that is only possible if it contains just one edge as $G$ is connected) or being adjacent is an equivalence relation, $G$ is a clique and we are done.
Assume for a contradiction that $a,b,c$ are vertices of $G$ as described above. $T$ contains a $P_4$ with vertices $u,v,w,z$ in this order such that $u$ is a leaf. Let $T'$ be the tree we obtain by deleting $u$ from $T$. First we will embed $T'$ into $G$ in a special way. 

Let us map $v$ to $a$, $w$ to $b$, and $z$ to $c$. We greedily embed the remaining vertices of $T'$ one by one into $G$. We always pick a vertex $x$ that is a leaf of the current, always increasing tree. This is doable, as the already embedded neighbor of $x$ has degree $t-2$ in $G$, and there are at most $t-2$ vertices of $T'$ already embedded into $G$, thus we can pick a new vertex for $x$. Finally, we embed $u$. Its only neighbor $v$ is already embedded into $a$. There are $t-2$ other vertices of $T'$ already embedded, and the degree of $a$ is $t-2$. However, $c$ is not adjacent to $a$ and is among the already embedded vertices, thus $a$ has a neighbor that is not in the copy of $T'$ we found in $G$. We embed $u$ into that neighbor to obtain a copy of $T$, a contradiction that finishes the proof.
\end{proof}

We need another lemma which describes the $T$-free $(t-3)$-regular graphs. It is similar to the above lemma, and its proof is also similar, but it is more involved.

We say that a tree is \textit{almost-star} if in its proper 2-coloring, one of the classes consists of at most two vertices. We denote by $A_t$ the almost-star on $t$ vertices where the class consisting of two vertices has a vertex of degree 2 and a vertex of degree $t-3$. We say that a tree is a \textit{double-star} if it has exactly two vertices of degree greater than 1.

\begin{lemma}\label{tminusthree}
Let\/ $t$ be even,\/ $T$ be a non-star tree on\/ $t$ vertices, and\/ $G$ be a\/ $(t-3)$-regular\/ $T$-free graph. Then there are three possibilities.

$\bullet$ If\/ $T=A_t$, then\/ $G$ is the vertex-disjoint union of copies of\/ $K_{t-2}$ and\/ $K_{t-3,t-3}$ and for integers\/ $k$ that divide\/ $t-3$, the complete\/ $(k+1)$-partite graph with parts of size\/ $(t-3)/k$.

$\bullet$ If\/ $T$ is an almost-star different from\/ $A_t$, then\/ $G$ is the vertex-disjoint union of copies of\/ $K_{t-2}$ and\/ $K_{t-3,t-3}$.

$\bullet$ If\/ $T$ is not an almost-star, then\/ $G$ is the vertex-disjoint union of copies of\/ $K_{t-2}$.
%For even $t$ and any tree $T$ on $t$ vertices that is not an almost-star, the only $T$-free $(t-3)$-regular graphs are the vertex-disjoint unions of copies of $K_{t-2}$. If $T$ is an almost-star, but not a star, then the only $T$-free $(t-3)$-regular graphs are the 
\end{lemma}

\begin{proof} %\vm{én erről az egészrol azt gondolom, hogy valahogy jobban beagyazas nyelven kellene leirni és nem by contradiction. Ha ez a gráf, bele tudunk agyazni ezt meg azt. }
Let us assume $G$ is connected, but not a clique. Observe first that the case $t<6$ is trivial. Our aim is to define an embedding $f:T\rightarrow G$ unless $G$ and $T$ are as in one of the cases enumerated in the statement of the lemma.

\vspace{3mm}

CASE 1. $T$ is a double star.  

\vspace{2mm}

Let $T$ be a double star with vertices $x,y$ having degree greater than 1, and $p$ leaves are adjacent to $x$ and $q$ leaves are adjacent to $y$. If in $G$ there are two adjacent vertices $u$ and $v$ such that they share at most $t-6$ neighbors, i.e.\ $u$ has $r\ge 2$ neighbors nonadjacent to $v$ (besides $v$), then we can define $f$ as follows. Observe that there are at least $t-2$ vertices adjacent to $u$ or $v$ (besides $u$ and $v$). First we let $f(x)=u$ and $f(y)=v$. Then we pick $\min\{r,p\}$ neighbors of $u$ that are not adjacent to $v$, and some additional vertices from the common neighborhood of $u$ and $v$ to obtain $p$ vertices and let them be the $f$-images of the leaves adjacent to $x$. Then $v$ has at least $t-2-p= q$  neighbors that have not yet appeared as $f$-images, thus we can define $f$ on the leaves adjacent to $y$ by letting these unused neighbors of $v$ to be the $f$-images. The definition of $f$ is complete.

Therefore, we can assume that the neighborhoods of any two adjacent vertices $u$ and $v$ of $G$ differ in at most one vertex (besides $u$ and $v$). Let $u$ be an arbitrary vertex and $A$ be the set of its neighbors. Observe that $|A|=t-3$ is odd, thus $|A|\ge 3$. Since $A$ and $u$ do not induce a clique, there is a vertex $v\in A$ that is adjacent to a $w\not\in A$.

We claim that $w$ is adjacent to every vertex of $A$. %\vm{we know, since we are in this subcase} 
By the assumption on the number of common neighbors of adjacent vertices, we know that all the other neighbors of $v$ are in $A\cup \{u\}$. In particular, $v$ is adjacent to every vertex of $A$ except for one, $v'$. As the neighborhoods of $v$ and $w$ differ only in $u$ (and in a vertex adjacent to $w$ but not to $v$), every other neighbor of $v$ is adjacent to $w$. In particular, $w$ is adjacent to a third vertex of $A$, $v''$. Observe that $v''$ is adjacent to $v'$, otherwise the neighborhood of $v'$ and $u$ would differ by at least two vertices. But then we can repeat the same argument with $v''$ in place of $v$, to obtain that its neighbor outside $A\cup \{u\}$ (which is $w$) is adjacent to all the neighbors of $v''$ inside $A$, which include $v'$.

Therefore, every vertex in $A$ is adjacent to $u$, $w$ and $t-5$ vertices in $A$. Hence $G[A]$ is a $(t-5)$-regular graph on $t-3$ vertices, a contradiction as both $t-5$ and $t-3$ are odd. This finishes the proof in case $T$ is a double-star. %\vm{contradicting to Lemma 2.2.? d ezt ide kéne írni.}

\vskip 0.3truecm

CASE 2. $T$ is not a double-star and $G$ is a complete multipartite graph.

\vskip 0.2truecm

CASE 2.1. $T$ is an almost-star.

\vskip 0.2truecm

If $G$ is bipartite, then it is a $K_{t-3,t-3}$ that does not contain any almost-star as $t-2$ vertices should belong to the same part. Therefore, we can assume that $G$ has more than two parts. It is obvious that the regularity implies each part has the same size $(t-3)/k$ for some $k$ and then $G$ is complete $(k+1)$-partite. Let $x$ and $y$ be the two vertices in the same class of the 2-coloring of $T$. For $1\le j \le k+1$ and $1\le i\le (t-3)/k$, let $v^j_i$ denote the $i$th vertex of the $j$th part. If $T\neq A_t$, we let $f(x)=v^1_1$  and $f(y)=v^2_1$ and for the common neighbor $z$ of $x$ and $y$ we let $f(z)=v^3_1$. Let $x_1,x_2,\dots, x_p$ be the leaves adjacent to $x$ and let $y_1,y_2,\dots, y_q$ be the leaves adjacent to $y$. Clearly, $t-3=p+q$ holds and $p=1$ or $q=1$ is equivalent to $T=A_t$. We write $p'=\min\{p,(t-3)/k-1\}$ and $q'=\min\{q,(t-1)/k-1\}$. For $1\le i \le p'$ we let $f(x_i)=v^2_{i+1}$ and for $1\le j \le q'$ we let $f(y_j)=v^1_{j+1}$. Finally, we try to use vertices $v^j_i$ with $j\ge 3$ as $f$-images of those vertices of $T$ for which $f$ has not yet been defined (apart from $v^3_1=f(z)$). So the number of available vertices of $G$ is $(k-1)(t-3)/k-1$ and the number of vertices of $T$ still to be embedded is $(p-p')+(q-q')=t-3-p'-q'$. Observe that as $k\ge 2$, we must have $\max\{p',q'\}\ge (t-3)/k-1$ and as the parts have size at least three, we must have $\min\{p',q'\}\ge 2$. Therefore, the number of vertices of $T$ still to be embedded is $t-3-p'-q'\le (k-1)(t-3)/k+1-2$. So it is indeed possible to define $f$.

If $T=A_t$, then the above embedding does not work, as then $\min\{p',q'\}=1$. Observe that one of $x$ and $y$, say $x$, has degree $t-3$ in $T$. Therefore, if $f$ is an embedding of $T$ into $G$, then $f(x_1),f(x_2),\dots,f(x_{t-3})$ cover all vertices of $G$ that are not in the same part as $f(x)$. That means vertices not adjacent to $x$ have to be in the same part, but there is an edge between them, so $f$ cannot exist, as stated.

\vskip 0.2truecm

CASE 2.2. $T$ is not an almost-star. 

\vspace{2mm}

As $T$ is not a double-star, it contains a $P_5$ with vertices $u,v,w,z,y$ in this order such that $u$ is a leaf. We will show that there is an embedding  $f:T\rightarrow G$ unless $G$ is among the possibilities listed in the lemma.
 If $G$ is bipartite, it is $K_{t-3,t-3}$, which contains every tree $T$ with both color classes of size at most $t-3$, i.e.\ every tree $T$ that is not an almost-star. 
Therefore, $G$ has more than two parts and we will keep the notation $v^j_i$ to denote the $i$th vertex of the $j$th part. The regularity of $G$ implies each part has the same size. 
Odd regularity implies we have an even number of parts (thus at least 4), and their size is odd. If the size of each part is 1, we have a clique and $K_{t-2}$ is listed among the possibilities in the statement of the lemma. Thus we can assume that each part of $G$ has size at least 3.

If $v$ has another leaf neighbor $u'$, we first let $f(v)=v^1_1$, $f(z)=v^1_2$ , and $f(w)=v^2_1$, $f(y)=v^2_2$. Then we fix an ordering $s_1,s_2,\dots, s_{t-6}$ of the vertices of $V(T)\setminus \{u,u',v,z,w,y\}$ such that $T[v,z,w,y,s_1,\dots, s_j]$ is a tree for any $1\le j \le t-6$, i.e.\ $s_j$ has exactly one neighbor in $\{v,z,w,y,s_1,\dots,s_{j-1}\}$. We then define $f(s_1),f(s_2),\dots,f(s_{t-6})$ one by one in this order, as follows. When we define $f(s_j)$, let $x$ be the unique neighbor of $s_j$ in $v,z,w,y,s_1,\dots,s_{j-1}$; then we let $f(s_j)$ to be a neighbor of $f(x)$ preferably in the first part, and arbitrarily if that is not possible (either because $f(x)$ is in the first part, or because all vertices of the first part are already $f$-images). We can do that, as $f(x)$ has degree $t-3$ in $G$, and even for $j=t-6$, only $t-3$ vertices of $T$ have already been embedded, including $x$. Thus $f(x)$ has a neighbor that is not an $f$-image, that we can use.
If we did not embed any other vertex into the first part (besides $v$ and $z$), then every vertex of $T$ is incident to $v$ or $z$, thus $T$ is an almost-star. Otherwise, we will be able to define $f(u)$ and $f(u')$. Indeed, $f(v)$ has degree $t-3$ and out of the already embedded $t-2$ vertices, at least three are in the first part, thus $f(v)$ has at least two unused neighbors, finishing the definition of $f$ in this case.

If $u$ is the only leaf vertex adjacent to $v$, we will pick $u$ and another leaf $u'$ with neighbor $v'\neq v$. Before describing how to pick them, we describe how we will use them. Let $T'$ be the tree obtained by deleting $u$ and $u'$ from $T$. We will define $f$ on the vertices of $T'$ first, in one of the following two ways: either $f(v)$ and $f(v')$ are in the first part together with a third vertex, or $f(v)$ and $f(v')$ are in the first and second part of $G$, at least three vertices are embedded into the first part, and at least two vertices are embedded into the second part. Observe that if we partially define $f$ on some subtree $T'' \subseteq T'$ where $f(v)$ and $f(v')$ already satisfy the above properties, than we can greedily extend $f$ to $T'$, as the minimum degree is $|V(T')|-1$.

Let us describe first how we define $f(u)$ and $f(u')$ once $f$ is defined on $T'$. In the first case ($(f(v),f(v')$ are both in the first part), we can define $f(u)$ and $f(u')$, as both $f(v)$ and $f(v')$ have $t-3$ neighbors in $G$, and there are at most $t-5$ vertices of $T'$ embedded in their neighborhood so far.
In the other case, we can define $f(u')$, as $f(v')$ has $t-3$ neighbors (in $G$), and we had embedded $t-2$ vertices, but two of them into the second part. Similarly, afterwards we can embed $f(u)$, as $f(v)$ has degree $t-3$ in $G$ and out of the already embedded $t-1$ vertices, at least three are in the first part, thus $f(v)$ has an unused neighbor, finishing the definition of $f$ in this case.

%If $w$ has a leaf neighbor, we embed $v$ and $y$ into the first part, while $w$ and $z$ in the second part. Then proceed as in the previous case, we embed the other vertices greedily such that we embed one in the first part if possible. This way if no other vertex is embedded in the first part, $T$ is a almost-star, and otherwise we are done.

%If $z$ has a leaf neighbor, we embed $v$ and $y$ into the first part, $z$ into the second part and $w$ into the third part. Then again, we embed the remaining vertices greedily such that we embed one in the first and then one in the second part if possible. 
%\pb EZT HOGY? $w$ ES $z$ SZOMSZEDOK!

%If none of $w$ or $z$ has a leaf neighbor, 
Let us return to defining and embedding $T'$. We extend the path $uvwzy$ to a maximal path $a_1a_2\dots a_ma_{m+1}$ of $T$; let $u'=a_{m+1}$, thus $u'$ is a leaf. We define $f$ on the non-leaf vertices $a_2=v,a_3=w,a_4=z,a_5=y,\dots,a_m$ of this path.

If $6\le m \le 2n/k+1$, then for $2\le h \le m$ we let $f(a_h)=v^1_{h/2}$ if $h$ is even, and $f(a_h)=v^2_{(h-1)/2}$ if $h$ is odd. In this case, $m\ge 6$ implies that we have defined $f$ for at least five vertices, and we have found the desired embedding of $T'$ and we are done. 

If $m=4$ or $m=5$ (i.e.\ the last non-leaf vertex of the path $a_m$ is $z$ or $y$), then again we define the $f$-image of $a_2,\dots,a_m$ alternately from the first and second part, and then we define $f$ on the remaining vertices greedily, in such a way that we embed them into the first part if possible. If we did not embed any other vertex into the first part, then each vertex is adjacent to $a_2$ or $a_4$, thus $T$ is an almost-star, in any other cases we get the desired embedding $f$ on $T'$.

Finally, if the whole path does not fit into the first and second parts, i.e.\ if $m-1 >2n/k$, then we embed only  $a_2,a_3, \dots, a_{2n/k}$ into the first and second part in an alternating way. Then we embed the remaining part of the path till $a_m$ alternately into the third and fourth part of $G$, then the fifth and sixth part, and so on. This way we could use all the vertices of $G$, which is obviously enough. Then we embed $a_m$ into the second part, and we obtain the desired partial embedding of $T'$.

%depending on the parity of the length of this cycle, we embed either $v$ and $z$ to the first part and $w$ and $y$ to the second part, or $v$ and $y$ to the first part and $w$, $z$ to the second part. In both cases we alternately place the next vertices into the first and second part. Afterwards, we proceed as in the earlier cases: we embed the other vertices, except for $u$ and $x$ greedily, in an order such that we always extend a tree. The only restriction is that we embed a vertex into the first part if possible. As in the earlier cases, if no other vertices are embedded in the first part, then $T$ is an almost-star, and otherwise we can extend the embedding greedily first by $x$, then by $u$. 

\vskip 0.3truecm

CASE 3. $G$ is not a complete multipartite graph. Recall that $T$ contains a $P_5$ with vertices $u,v,w,z,y$ in this order such that $u$ is a leaf.

\vskip 0.2truecm

Our aim is to define an embedding $f:T\rightarrow G$ vertex by vertex in a specific order $s_1,s_2,\dots,s_t$. To define the embedding we will need the following claim.

\begin{claim}\label{elso} If $t\ge 7$, then $G$ contains a $P_4$ $abcd$ such that $ac,ad\not\in E(G)$, and 
%either $bd\not\in E(G)$, or 
there is no $K_{t-6}$ consisting of other vertices such that all of its vertices are adjacent to all of $a,b,c,d$.
\end{claim}

\begin{proof}
Observe that if for any two nonadjacent vertices $\alpha,\beta \in V(G)$ we have $N(\alpha)=N(\beta)$, then being nonadjacent is an equivalence relation, i.e.\ $G$ is complete multipartite, contradicting to the assumption of CASE 3.

Also, if for any two nonadjacent vertices $\alpha,\beta \in V(G)$ we have that $N(\alpha)\neq N(\beta)$ implies $N(\alpha)\cap N(\beta)=\emptyset$, then being adjacent or equal is an equivalence relation, and thus $G$ is a vertex-disjoint union of cliques, and we are done.

Otherwise, we can find two nonadjacent vertices $\alpha$ and $\beta$ such that $C:=N(\alpha)\cap N(\beta)$, $A:=N(\alpha)\setminus N(\beta)$, and $B:=N(\beta)\setminus N(\alpha)$ are all non-empty, as $G$ is regular. 
%If there are $u\in A$ and $v\in B\cup C$ such that $uv\not\in E(G)$, they form the first desired configuration with $x$ and $y$. Otherwise we pick arbitrary 
Let us fix $\gamma\in C$, $\delta\in B$,  $\epsilon \in A$ and consider the $P_4$ $\alpha \gamma \beta \delta$. If $M:=N(\alpha)\cap N(\beta)\cap N(\gamma)\cap N(\delta)$ does not contain a clique of size $t-6$, then we can take $a=\alpha$, $b=\gamma$, $c=\beta$, and $d=\delta$, finishing the proof. 

Observe that if $K\subseteq M$ is a clique of size $t-6$, then for any $\kappa \in V(K)$ we have $N(\kappa) = (K\setminus \{\kappa\}) \cup \{\alpha,\beta,\gamma,\delta\}$. Indeed, as $\kappa$ is adjacent to all these vertices and their number is $t-3$, such $\kappa$ is not adjacent to $\epsilon$. Therefore, $N:=N(\alpha)\cap N(\beta)\cap N(\gamma)\cap N(\epsilon)$ is disjoint from $K\cup \{\gamma,\delta\}$. As $K\cup \{\gamma,\delta\}\subset N(\beta)$, we obtain $|N|\le 1$ and $a=\beta, b=\gamma, c=\alpha, d=\epsilon$ is a good choice if $t\ge 8$. Finally, if $t=7$ and $K=\{\kappa\}$, $N=\{\kappa'\}$, then $N(\kappa)=\{\alpha,\beta,\gamma,\delta\}$, $N(\gamma)=\{\alpha,\beta,\kappa,\kappa'\}$, $N(\beta)=\{\gamma,\delta,\kappa,\kappa'\}$, thus $a=\beta$, $b=\kappa$, $c=\alpha$, $d=\epsilon$ is a good choice.
%If there is a $K_{t-6}$ as described in the statement, it consists of elements of $A\cup B\cup C$. Notice that the clique and the path contains all the neighbors of every vertex in the clique. In particular, an element of $A$ in the clique means every element of $B\cup C\setminus \{v\}$ is in the clique, which in turn means every element in $A\setminus \{u\}$ is in the clique. Similarly, an element of $B\cup C \setminus \{v\}$ in the clique forces every element of $A\cup B\cup C\setminus \{u,v\}$ to be in the clique. But $|A\cup B\cup C\setminus \{u,v\}|\ge t-2$, thus the vertices of the clique have degree more than $t-3$, a contradiction.
\end{proof}

Let $T_0$ be the tree obtained from $T$ by removing its leaves. As $T$ contains a path on 5 vertices, $T_0$ contains a path on 3 vertices.  Suppose $T_0$ contains a leaf $v$ such that $v$ is adjacent to two leaves $u_1,u_2$ of $T$. Then there exists a path $vwz$ in $T_0$ and thus a path $vwzy$ in $T$. Let us fix an ordering $s_1,s_2,\dots,s_t$ of the vertices of $T$ such that $s_1=v$, $s_2=w$, $s_3=z$, $s_4=y$, $s_{t-1}=u_1,s_{t-2}=u_2$ and for any $j\le t$ the subgraph $T[s_1,s_2,\dots,s_j]$ is connected, i.e.\ every $s_j$ ($j\ge 2$) is adjacent to exactly one $s_i$ with $i<j$. We let $f(s_1)=a$, $f(s_2)=b$, $f(s_3)=c$, $f(s_4)=d$ with $a,b,c,d$ given by Claim \ref{elso}. As earlier, we extend the definition of $f$ vertex by vertex; that is, when defining $f(s_j)$, we take a neighbor of $f(s_{i})$ where $s_i$ is the unique neighbor of $s_j$ in $T$. This is clearly doable as long as $j\le t-2$. Indeed, at that moment $t-3$ vertices are embedded and $f(s_i)$ is not a neighbor of itself, so at least one neighbor is available. As $s_{t-1},s_t$ are the leaves adjacent to $s_1$, the only question is whether $f(s_1)=a$ still has two available neighbors. It does, as the number of already embedded vertices is $t-2$ and $a$ is not adjacent to itself, $f(s_3)=c$, and $f(s_4)=d$.

Suppose next that all leaves of $T_0$ are adjacent to exactly one leaf of $T$. Let $v_1v_2\dots v_k$ be a longest path in $T_0$ and let $u_1$ and $u_k$ denote the leaves of $T$ adjacent to $v_1$ and $v_k$, respectively. Let us fix a numbering $s_1,s_2,\dots,s_t$ of the vertices of $T$ as follows.

    If $k\ge 5$, $s_j=v_j$ for $j=1,2,3, 4$, $s_t=u_1$, $s_{t-1}=u_k, s_{t-2}=v_k$  and for any $j\le t$ the subgraph $T[s_1,s_2,\dots,s_j]$ is connected, i.e.\ every $s_j$ ($j\ge 2$) is adjacent to exactly one $s_i$ with $i<j$. (Note that such a numbering exists because $v_k$ is adjacent to only one leaf in $T_0$, namely $u_k$.) We let $f(s_1)=a$, $f(s_2)=b$, $f(s_3)=c$, $f(s_4)=d$ with $a,b,c,d$ given by Claim \ref{elso}, and then consider a maximal clique $K$ in $G$ with all vertices of $K$ adjacent to all of $a,b,c$, and $d$. By Claim \ref{elso}, we have $m:=|M|\le t-7$. We let $f$ map the vertices $s_5,s_6,\dots,s_{m+4}$ to vertices of $M$. Then for $m+5\le j \le t-2$, we can define $f(s_j)$ to be an arbitrary available neighbor of $f(s_i)$, where $s_i$ is the single vertex adjacent to $s_j$ with $i<j$. Observe that $m\le t-7$ implies that $f(s_{t-2})\notin M$, therefore $f(s_{t-2})$ is not adjacent to at least one $f(s_i)$ with $i<t-2$. Thus $f(s_{t-2})$ still has an available neighbor that can serve as $f(s_{t-1})$. Finally, $f(s_t)$  should be an available neighbor of $f(s_1)$ and there exists such a neighbor as, by Claim \ref{elso}, $f(s_1)$ is not adjacent to $f(s_3)$ and $f(s_4)$.

    If $k=4$, then let us define the numbering  as $s_j=v_j$ for $j=1,2,3,4$, $s_t=u_1$, $s_{t-1}=u_k$,  and for any $j\le t$ the subgraph $T[s_1,s_2,\dots,s_j]$ is connected, i.e.\ every $s_j$ ($j\ge 2$) is adjacent to exactly one $s_i$ with $i<j$. Then the strategy of the case $k\ge 5$ works in this case as well, because $s_{t-1}$ is adjacent to $s_4$ and $f(s_4)=d$ is not adjacent to $f(s_1)=a$, so $f(s_4)$ will have an available neighbor when embedding $s_{t-1}$. Also, $f(s_1)=a$ is not adjacent to $c$ and $d$, therefore it will have an available neighbor when embedding $s_t$.

    Finally, if $k=3$, then $T_0$ is a star with, say, $l$ leaves. As any leaf of $T_0$ is adjacent to exactly one leaf of $T$, we obtain that $T$ has $2l+1$ vertices, contradicting our assumption that $t$ is even. %$v,v_{1,1},v_{1,2},v_{2,1},v_{2,2,},\dots,v_{l,1},v_{l,2}$ with $vv_{i,1}v_{i,2}$ forming a $P_3$ for every $i=1,2,\dots,l$. As $t\ge 6$ we obtain $\le 3$. Let us define the numbering as $s_1=v_{1,1},s_2=v,s_3=v_{2,1},s_4=v_{2,2},\dots, s_{2l-1}=v_{l,1},s_{2l}=v_{l,2},s_{2l+1}=v_{1,2}$. As $s_{2l}$ is connected to $s_{2l-1}$ and $s_{2l+1}$ is connected to $s_1$ the strategy of the case $k\ge 5$ works.
\end{proof}

\begin{proof}[Proof of Theorem \ref{treethm}] The upper bound $t-2$ is immediate from the well-known fact that a graph $G$ with minimum degree at least $t-1$ contains $T$. It can be proved by greedily embedding $T$ into $G$. It is sharp in case $t-1$ divides $n$, as shown by the vertex-disjoint union of $n/(t-1)$ copies of $K_{t-1}$. It is also sharp in case $T$ is a star and $t$ or $n$ is even, as shown by any $(t-2)$-regular graph.

In case $T$ is a star, and both $n$ and $t$ are odd, there is no $(t-2)$-regular graph on $n$ vertices, giving the upper bound $t-3$, which is sharp, as shown by any $(t-3)$-regular graph. If $T$ is not a star, then by Lemma \ref{nemklikk} the only $(t-2)$-regular $T$-free graphs consist of copies of $K_{t-1}$, thus $t-1$ divides $n$, giving the upper bound $t-3$ in every other case.

If $t$ is odd, then let $H$ be obtained by deleting a perfect matching from $K_{t-1}$. Then $H$ is $(t-3)$-regular, and so is $K_{t-2}$. If $n$ can be written as $a(t-1)+b(t-2)$ (every $n$ that is large enough can be written this way for some nonnegative integers $a$ and $b$), then $a$ copies of $H$ and $b$ copies of $K_{t-2}$ form a $(t-3)$-regular $T$-free graph on $n$ vertices. If $t-2$ divides $n$, then $n/(t-2)$ copies of $K_{t-2}$ form a $(t-3)$-regular $T$-free graph. If $T$ is an almost-star, and $t$ is even, then we write $n$ as $a(t-2)+b(2t-6)$ (we can do that if $n$ is large enough and even). Then $a$ copies of $K_{t-2}$ and $b$ copies of $K_{t-3,t-3}$ form a $T$-free graph on $n$ vertices (note that $K_{t-3,t-3}$ is $T$-free as in the proper 2-coloring of $T$, one class has size at least $t-2$). These constructions show that $t-3$ is a lower bound on $\regex(n,T)$ in all the claimed cases.

Finally, assume that the pair of $T$ and $n$ does not satisfy any of the above properties, thus we claim $\regex(n,T)=t-4$. We have already seen $\regex(n,T)<t-2$. Assume $\regex(n,T)=t-3$, then by the above $t$ is even, and by Lemma \ref{tminusthree} there are three possibilities. In the first and second cases $T$ is an almost-star, and $G$ is the union of vertex-disjoint copies of $K_{t-2}$ and some complete $(k+1)$-partite graphs for some $k$ that divides $t-3$. Then $k$ is odd, thus all of the above components have an even number of vertices, thus $n$ is even, and the pair of $T$ and $n$ was listed earlier. In the third case $t-2$ divides $n$, and again the pair of $T$ and $n$ was listed earlier. This proves the upper bound $t-4$ for the remaining cases. 

Finally, we prove the matching lower bound.
Let $H'$ be obtained by deleting a perfect matching from $K_{t-2}$ (recall that $t$ is even). Then $H'$ is $(t-4)$-regular, and so is $K_{t-3}$. Now we write $n$ in the form $a(t-2)+b(t-3)$, and we are done with the lower bound as in the previous cases, finishing the proof.
\end{proof}

\section*{Acknowledgement} 
Research was supported by the National Research, Development and Innovation Office -- NKFIH under the grants FK 132060, KH130371 and SNN 129364. Research of Vizer was supported by the J\'anos Bolyai Research Fellowship and the New National Excellence Program under the grant number \'UNKP-19-4-BME-287.

\end{document}